\documentclass[a4paper,11pt]{amsart}
\usepackage{mathrsfs,nccmath,amssymb,amsthm,mathabx}
\usepackage[dvips]{graphicx,psfrag}
\usepackage[all]{xy}
\xyoption{line}
\usepackage{color}
\usepackage{ulem}

\renewcommand{\le}{\varleq}
\renewcommand{\ge}{\vargeq}

\renewcommand{\theenumi}{\thesection.\arabic{enumi}}

\newcommand{\red}{\color{red}}

\newcommand\mL{L\kern-0.08cm\char39}

\newcommand{\myforall}{\text{ for all }}

\newcommand{\mywith}{\text{ with }}

\newcommand{\myand}{\text{ and }}
\newcommand{\myif}{\text{ if }}

\newcommand{\mythen}{\text{ then }}

\newcommand{\seb}{\{\,}
\newcommand{\sen}{\,\}}

\newcommand{\getsby}[1]{\xleftarrow{#1}}

\newcommand{\infd}{$(\to)$-depth\ }
\newcommand{\infde}{$(\to)$-depth}
\newcommand{\hx}{\hat{x}}
\newcommand{\hy}{\hat{y}}

\newcommand{\hX}{\hat{X}}

\newcommand{\tesgh}{edge-surjective graph homomorphism}

\newcommand{\pdirectional}{\raise0.05em\hbox{$+$}directional}
\newcommand{\bidirectional}{bidirectional}

\newcommand{\Z}{\mathbb{Z}}
\newcommand{\Nonne}{\mathbb{N}}
\newcommand{\Posint}{\mathbb{N} \setminus \{ 0 \}}

\newcommand{\bi}{\in \Z}

\newcommand{\beposint}{\in \Posint}
\newcommand{\bpi}{\ge 1} 

\newcommand{\benonne}{\in \Nonne}
\newcommand{\bni}{\ge 0} 

\newcommand{\num}[1]{\hash {#1} }

\newcommand{\Decomp}{\mathscr{D}}

\newcommand{\Gcal}{\mathcal{G}}

\newcommand{\Ucal}{\mathcal{U}}

\newcommand{\kuu}{\emptyset}
\newcommand{\nekuu}{\neq \kuu}
\newcommand{\fai}{\varphi}

\newcommand{\barc}{\bar{c}}

\newcommand{\barC}{\bar{C}}

\newcommand{\barX}{\bar{X}}

\newcommand{\ddC}{\ddot{C}}

\newcommand{\ddX}{\ddot{X}}

\newcommand{\enumb}{\begin{enumerate}}
\newcommand{\enumn}{\end{enumerate}}

\newcommand{\itemb}{\begin{itemize}}
\newcommand{\itemn}{\end{itemize}}

\renewcommand{\theenumi}{(\alph{enumi})}
\renewcommand{\labelenumi}{\theenumi}

\newtheorem{thm}{Theorem}[section]
\newtheorem{lem}[thm]{Lemma}
\newtheorem{prop}[thm]{Proposition}
\newtheorem{cor}[thm]{Corollary}

\theoremstyle{definition}
\newtheorem{defn}[thm]{Definition}

\theoremstyle{remark}
\newtheorem{nota}[thm]{Notation}
\newtheorem{rem}[thm]{Remark}

\numberwithin{equation}{section}

\newcommand{\Vp}{V \setminus V_0}

\newcommand{\ft}{$f$-tower}

\newcommand{\KR}{Kakutani--Rohlin}
\newcommand{\KRD}{\KR\ decomposition}
\newcommand{\AP}[1]{{\rm AP(#1)}}
\newcommand{\NAP}[1]{#1 \setminus \AP{#1}}

\begin{document}

\title[Non-homeomorphic topological rank]
{Non-homeomorphic topological rank and expansiveness}

\author{TAKASHI SHIMOMURA}

\address{Nagoya University of Economics, Uchikubo 61-1, Inuyama 484-8504, Japan}
\curraddr{}
\email{tkshimo@nagoya-ku.ac.jp}
\thanks{}



\date{\today}

\dedicatory{}

\commby{}

\begin{abstract}
Downarowicz and Maass (2008) have shown that every Cantor minimal homeomorphism
 with finite topological rank $K > 1$ is expansive.
Bezuglyi, Kwiatkowski and Medynets (2009) extended the result to non-minimal cases.
On the other hand, Gambaudo and Martens (2006) had expressed all Cantor minimal
 continuou surjections as the inverse limit of graph coverings.
In this paper, we define a topological rank for every Cantor minimal continuous
 surjection, and show that every Cantor minimal continuous
 surjection of finite topological rank has the natural extension that is expansive.
\end{abstract}

\maketitle
\section{Introduction}
By a {\it zero-dimensional system}, we mean a pair $(X,f)$
 of a compact zero-dimensional metrizable space $X$,
 and a continuous surjective map $f : X \to X$.
In \cite{Shimomura4}, we showed that every zero-dimensional sytem is expressed 
 as an inverse limit of a sequence of covers of finite directed graphs.
In this paper, instead of the term a `sequence of graph covers', we use the term
 a `graph covering' or just a `covering' in short.
On the other hand, for Cantor minimal homeomorphisms,
 properly ordered Bratteli diagrams have been
 used for deep investigations.
In \cite{DM}, Downarowicz and Maass presented a remarkable theorem that states that
every Cantor minimal system of finite topological rank $K > 1$ is expansive.
They used the properly ordered Bratteli diagrams with astonishing technique.
In \cite{BKM}, Bezuglyi, Kwiatkowski and Medynets extended the result
 to non-minimal cases.
In this paper, we want to extend the result of Downarowicz and Maass to non-homeomorphic case.
As for general Cantor minimal continuous surjections, Gambaudo and Martens in \cite{GM},
 presented a way of expressing every Cantor minimal
 continuous surjection by a kind of graph coverings.
Hereafter, we mention the coverings that Gambaudo and Martens gave in \cite{GM}
 as {\it GM-coverings}
 (see Definition \ref{defn:GM-covering} for details).
Applying the proof of \cite{DM} and GM-coverings, we show an analogous theorem
 for Cantor minimal continuous surjections.
We can define the topological rank for a Cantor minimal continuous surjection
 using GM-coverings.
Here, because we have extended the kind of coverings, logically, the topological rank
 may decrease from that of Downarowicz and Maass \cite{DM}.
Nevertheless, a simple observation shows that the topological rank is the same
 for Cantor minimal homeomorphisms (see Corollary \ref{cor:coincidence-rank}).
Our main result is as follows:
 for a Cantor minimal continuous surjection that has topological rank
 $K > 1$,
 the natural extension is expansive.
\section{Preliminaries}
Let $\Z$ denote the set of all integers;
 and $\Nonne$, the set of all non-negative integers.
In this section,
 we repeat the construction of general graph coverings
 for general zero-dimensional systems originally given in \S 3 of \cite{Shimomura4}.
We describe some notations for later use.
For $m \ge n$, we denote as $[n,m] := \seb n, n+1,\dotsc, m\sen$.
A pair $G = (V,E)$
 consisting of a finite set $V$ and a relation $E \subseteq V \times V$ on $V$
 can be considered as a directed graph with vertices $V$
 and an edge from $u$ to $v$ when $(u,v) \in E$.
Unlike Bratteli diagrams that are well known and defined in \S \ref{sec:bratteli}, multiple edges from a vertex $u$ to $v$ is not permitted.
We note here that
 the expression $(V,E)$ is also used to mean a Bratteli diagram in this paper.
If we write a ``graph $G$'', ``graph $G = (V,E)$'' or a ``surjective directed graph $G = (V,E)$'',
 we mean a finite directed graph.
When the expression $(V,E)$ mean a Bratteli diagram, we explicitly write as a ``Bratteli diagram $(V,E)$''.
%
%
\begin{nota}
In this paper, we assume that a finite directed graph $G$ is a surjective relation,
 i.e., for every vertex $v \in V$ there exist edges $(u_1,v),(v,u_2) \in E$.
\end{nota}
For directed graphs $G_i = (V_i,E_i)$ with $i = 1,2$,
 a map $\fai : V_1 \to V_2$ is said to be a {\it graph homomorphism}
 if for every edge $(u,v) \in E_1$, it follows that $(\fai(u),\fai(v)) \in E_2$.
In this case, we write as $\fai : G_1 \to G_2$.
For a graph homomorphism $\fai : G_1 \to G_2$, we say that $\fai$ is {\it edge-surjective}
if $\fai(E_1) = E_2$.
%
%
%
%
%
Suppose that a graph homomorphism $\fai : G_1 \to G_2$ satisfies the following condition:
\[(u,v),(u,v') \in E_1 \text{ implies that } \fai(v) = \fai(v').\]
In this case, $\fai$ is said to be {\it \pdirectional}.
Suppose that a graph homomorphism $\fai$ satisfies both of the following conditions:
\[(u,v),(u,v') \in E_1 \text{ implies that } \fai(v) = \fai(v') \myand \]
\[(u,v),(u',v) \in E_1 \text{ implies that } \fai(u) = \fai(u').\]
Then, $\fai$ is said to be {\it \bidirectional}.\\
%
%
\vspace{-3mm}

\begin{defn}\label{defn:cover}
A graph homomorphism $\fai : G_1 \to G_2$ is called a {\it cover}\/ if it is a \pdirectional\ \tesgh.
\end{defn}
For a sequence $G_1 \getsby{\fai_1} G_2 \getsby{\fai_2} \dotsb$ of graph homomorphisms
 and $m > n$, we write $\fai_{m,n} := \fai_{n} \circ \fai_{n+1} \circ \dotsb \circ \fai_{m-1}$.
Then, $\fai_{m,n}$ is a graph homomorphism.
If all ${\fai_i}$ $(i \beposint)$ are edge surjective, then every $\fai_{m,n}$ is edge surjective.
If all ${\fai_i}$ $(i \beposint)$ are covers, every $\fai_{m,n}$ is a cover.
%
%
Let $G_0 := \left(\seb v_0 \sen, \seb (v_0,v_0) \sen \right)$ be a singleton graph.
For a sequence of graph covers $G_1 \getsby{\fai_1} G_2 \getsby{\fai_2} \dotsb$, we
 attach the singleton graph $G_0$ at the head.
We call a sequence of graph covers
 $G_0 \getsby{\fai_0} G_1 \getsby{\fai_1} G_2 \getsby{\fai_2} \dotsb$
 as a {\it graph covering} or just a {\it covering}.
%
%
%
%
Let us write the directed graphs as $G_i = (V_i,E_i)$ for $i \benonne$.
Define
\[V_{\Gcal} := \seb (x_0,x_1,x_2,\dotsc) \in \prod_{i = 0}^{\infty}V_i~|~x_i = \fai_i(x_{i+1}) \text{ for all } i \benonne \sen \text{ and}\]
\[E_{\Gcal} := \seb (x,y) \in V_{\Gcal} \times V_{\Gcal}~|~(x_i,y_i) \in E_i \text{ for all } i \benonne\sen,\]
each equipped with the product topology.
\begin{nota}\label{nota:opensets-of-vertices}
For each $n \benonne$, the projection from $V_{\Gcal}$ to $V_n$ is denoted by $\fai_{\infty,n}$. 
For $v \in V_n$, we denote a clopen set $U(v) := \fai_{\infty}^{-1}(v)$.
For a subset $V \subset V_n$, we denote a clopen set $U(V) := \bigcup_{v \in V}U(v)$.
\end{nota}
\begin{nota}\label{nota:decomp}
Let $X$ be a compact metrizable zero-dimensional space.
A finite partition of $X$ by non-empty clopen sets is called a {\it decomposition}.
The set of all decompositions of $X$ is denoted by $\Decomp(X)$.
Each $\Ucal \in \Decomp(X)$ is endowed with the discrete topology.
\end{nota}
%
%
We can state the following:
\begin{thm}[Theorem 3.9 and Lemma 3.5 of \cite{Shimomura4}]\label{thm:0dim=covering}
Let $\Gcal$ be a covering
 $G_0 \getsby{\fai_0} G_1 \getsby{\fai_1} G_2 \getsby{\fai_2} \dotsb$.
Then, $E_{\Gcal}$ is a continuous surjective mapping and $(V_{\Gcal},E_{\Gcal})$ is
 a zero-dimensional system.
Conversely, every zero-dimensional system can be written in this way.
Furthermore, if all $\fai_n$ are \bidirectional, then this zero-dimensional system is a 
 homeomorphism and every compact zero-dimensional homeomorphism is written in this way.
\end{thm}
%
%
We write $(V_{\Gcal},E_{\Gcal})$ as $G_{\infty}$.
%
%
Take a subsequence $n_0 = 0 < n_1 < n_2 < \dotsb$.
Then, we can get an essentially same covering
\[ G_0 \getsby{\fai_{n_1,0}} G_{n_1} \getsby{\fai_{n_2,n_1}} G_{n_2} \dotsb.\]
It is evident that the new covering produces
 a naturally topologically conjugate zero-dimensional  system.
Following the naming in the theory of Bratteli--Vershik systems,
 we call this procedure as
 {\it telescoping}.
%
%
%
%
%
%
%
%
%
\begin{nota}
Let $G = (V,E)$ be a surjective directed graph.
A sequence of vertices $(v_0,v_1,\dotsc,v_l)$ of $G$ is said to be a {\it walk}\/ of length $l$ if $(v_i, v_{i+1}) \in E$ for all $0 \le i < l$.
We denote as $l(w) := l$.
We say that a walk $w = (v_0,v_1,\dotsc,v_l)$ is a {\it path}\/
 if $v_i$ $(0 \le i \le l)$ are mutually distinct.
A walk $c = (v_0,v_1,\dotsc,v_l)$ is said to be a {\it cycle}\/ of period $l$
 if $v_0 = v_l$,
 and a cycle $c = (v_0,v_1,\dotsc,v_l)$
 is a {\it circuit} of period $l$ if the $v_i$ $(0 \le i < l)$ are mutually distinct.
A circuit $c$ and a path $p$ are also considered to be subgraphs of $G$
 with period $l(c)$ and length $l(p)$ respectively.
For a walk $w = (v_0,v_1,\dotsc,v_l)$,
 we define as $V(w) := \seb v_i \mid 0 \le i \le l \sen$
 and $E(w) := \seb (v_i,v_{i+1}) \mid 0 \le i < l \sen$.
For a subgraph $G'$ of $G$, we also define $V(G')$ and $E(G')$ in the same manner,
 especially, $V(G) = V$ and $E(G) = E$.
\end{nota}
%
%
\begin{nota}
Let $w_1 = (u_0,u_1,\dotsc, u_l)$ and $w_2 = (v_0,v_1,\dotsc, v_{l'})$ be walks
 such that $u_l = v_0$.
Then, we denote $w_1  w_2 := (u_0,u_1,\dots,u_l,v_1,v_2,\dotsc,v_{l'})$.
Evidently, we get $l(w_1  w_2) = l+l'$.
If $c$ is a cycle of length $l$, then for any positive integer $n$,
 a cycle $c^n$ of length $ln$ is well defined.
When a walk $w  c  w'$ can be constructed, a walk $w c^0  w'$ means $w  w'$.
\end{nota}

\section{Bratteli--Vershik systems}\label{sec:bratteli}
We follow Gjerde and Johansen \cite[\S 1]{GJ} to describe the Bratteli--Vershik representation for Cantor minimal homeomorphisms.
\begin{defn}
A {\it Bratteli diagram} is an infinite directed graph $(V,E)$, where $V$ is the vertex
set and $E$ is the edge set.
These sets are partitioned into non-empty disjoint finite sets 
$V = V_0 \cup V_1 \cup V_2 \cup \dotsb$ and $E = E_1 \cup E_2 \cup \dotsb$,
 where $V_0 = \seb v_0 \sen$ is a one-point set.
Each $E_n$ is a set of edges from $V_{n-1}$ to $V_n$.
Therefore, there exist two maps $r,s : E \to V$ such that $r:E_n \to V_n$
 and $s : E_n \to V_{n-1}$ for $n \bpi$, the range map and the source map respectively.
Moreover, $s^{-1}(v) \nekuu$ for all $v \in V$ and
$r^{-1}(v) \nekuu$ for all $v \in V \setminus V_0$.
We say that $u \in V_{n-1}$ is connected to $v \in V_{n}$ if there
 exists an edge $e \in E_n$ such that $s(e) = u$ and $r(e) = v$.
Unlike graph coverings,
 multiple edges between $u$ and $v$ are permitted.
The {\it rank $K$} of a Bratteli diagram is defined as
 $K := \liminf_{n \to \infty}\hash V_n$,
 where $\hash V_n$ is the number of elements in $V_n$.
\end{defn}
%
%
Let $(V,E)$ be a Bratteli diagram and $m < n$ be non-negative integers.
We define
\[E_{m,n} := \seb p \mid p \text{ is a path from a } u \in V_m \text{ to } v \in V_n \sen.\]
Then, we can construct a new Bratteli diagram $(V',E')$ as follows:
\[ V' := V_0 \cup V_1 \cup \dotsb \cup V_m \cup V_n \cup V_{n+1} \cup \dotsb \]
\[ E' := E_1 \cup E_2 \cup \dotsb \cup E_m \cup E_{m,n} \cup E_{n+1} \cup \dotsb. \]
The source map and the range map are also defined naturally.
This procedure is called {\it telescoping}.

\begin{defn}
A Bratteli diagram is called {\it simple}
 if we make (at most countably many times of) telescopings,
 all pairs of vertices $u \in V_n$ and $v \in V_{n+1}$ are joined
 by at least one edge for all $n \ge 0$.
\end{defn}

\begin{defn}
Let $(V,E)$ be a Bratteli diagram such that
$V = V_0 \cup V_1 \cup V_2 \cup \dotsb$ and $E = E_1 \cup E_2 \cup \dotsb$
 are the partitions,
 where $V_0 = \seb v_0 \sen$ is a one-point set.
Let $r,s : E \to V$ be the range map and the source map respectively.
We say that $(V,E,\le)$ is an {\it ordered}\/ Bratteli diagram if
 the partial order $\le$ is defined on $E$ such that 
 $e, e' \in E$ is comparable if and only if $r(e) = r(e')$.
 That is, we have a linear order on each set $r^{-1}(v)$ with $v \in \Vp$.
The edges $r^{-1}(v)$ are numbered from $1$ to $\hash(r^{-1}(v))$.
\end{defn}
Let $n > 0$ and $e = (e_n,e_{n+1},e_{n+2},\dotsc), e'=(e'_n,e'_{n+1},e'_{n+2},\dotsc)$ be cofinal paths from vertices of $V_{n-1}$, which might be different.
We get the lexicographic order $e < e'$ as follows:
\[\myif k \ge n \text{ is the largest number such that } e_k \ne e'_k, \mythen e_k < e'_k.\]%
\begin{defn}
Let $(V,E,\le)$ be an ordered Bratteli diagram.
Let $E_{\max}$ and $E_{\min}$ denote the set of maximal and minimal
edges, respectively.
An infinite path is maximal (minimal) if all the edges making up the
path are elements of $E_{\max}$ ($E_{\min}$).
\end{defn}

\begin{defn}
An ordered Bratteli diagram is properly ordered if it is
simple and if it has a unique maximal and a unique minimal path, denoted respectively by
$x_{\max}$ and $x_{\min}$.
\end{defn}

\begin{defn}[Vershik map]
Let $(V,E,\le)$ be a properly ordered Bratteli diagram.
Let 
\[E_{0,\infty} := \seb (e_1,e_2,\dotsc) \mid r(e_i) = s(e_{i+1})
 \myforall i \ge 1 \sen,\]
with the subspace topology of the product space $\prod_{i = 1}^{\infty}E_i$.
We can define the {\it Vershik map} $\phi : E_{0,\infty} \to E_{0,\infty}$ as follows:

\noindent If $e = (e_1,e_2,\dotsc) \ne x_{\max}$, then there exists the least $n \ge 1$ such that 
 $e_n$ is not maximal in $r^{-1}(r(e_n))$.
Then, we can select the least $f_n > e_n$ in $r^{-1}(r(e_n))$.
Let $v_{n-1} = s(f_n)$.
Then, it is easy to get the unique least path $(f_1,f_2,\dotsc,f_{n-1})$
 from $v_0$ to $v_{n-1}$.
We define $\phi(e) := (f_1,f_2,\dotsc,f_{n-1},f_n,e_{n+1},e_{n+2},\dotsc)$.
We define as $\phi(x_{\max}) = x_{\min}$.
The map $\phi : E_{0,\infty} \to E_{0,\infty}$ is called the {\it Vershik map}.
\end{defn}
In a theorem \cite[Theorem 4.7]{HPS}, we can find a correspondence
 that for a properly ordered Bratteli diagram,
 the Vershik map is a minimal homeomorphism
 from a compact metrizable zero-dimensional space to itself;
 conversely, a minimal homeomorphism from
 a compact metrizable zero-dimensional space to itself
 is represented as the Vershik map of a properly ordered Bratteli diagram.

In \cite{DM}, Downarowicz and Maass introduced the topological rank for
 a Cantor minimal homeomorphism.
\begin{defn}
Let $(X,f)$ be a Cantor minimal homeomorphism.
Then, the topological rank of $(X,f)$ is $0 \le K \le \infty$,
 if it has a Bratteli--Vershik representation with Bratteli diagram with rank $K$,
 and $K$ is the minimal of such numbers.
\end{defn}%
\section{Covering of Kakutani--Rohlin Type}
%
%
Before we proceed to the coverings for general minimal continuous surjections that
are described in \S \ref{sec:gambaudo-martens}, here we introduce a type of coverings
that is closely related to Bratteli diagrams.
We shall construct a covering
 $G_0 \getsby{\fai_0} G_1 \getsby{\fai_1} G_2 \getsby{\fai_2} \dotsb$
that can express a certain set of Cantor minimal continuous surjections that
 contains all Cantor minimal homeomorphisms but not all minimal continuous surjections.
Let $G_0$ be the singleton graph, that corresponds to
 the top vertex of the Bratteli diagram.
We shall construct graphs $G_n$ $(n \ge 1)$ that are generalized figure-8s.
The generalized figure-8s are finite directed graphs
 that have a unique central vertex, and distinct
 circuits that start and end at the central vertex.
We denote the central vertex as $v_{n,0}$, and distinct circuits as 
 $\seb c_{n,1},c_{n,2},\dotsc,c_{n,l_n}\sen$.
We write the period of each circuit $c_{n,i}$ with $1 \le i \le l_n$
 as $l(n,i) \ge 1$.
Thus, we write
\[c_{n,i}
 = (v_{n,i,0} = v_{n,0}, v_{n,i,1}, v_{n,i,2}, \dotsc, v_{n,i,l(n,i)} = v_{n,0})\]
for $n \ge 1$ and $1 \le i \le l_n$.
As we are constructing each $G_n$ as a generalized figure-8, we assume that
 $V(c_{n,i}) \cap V(c_{n,j}) = \seb v_{n,0} \sen$
 for all $n \ge 1$ and $i \ne j$.

\begin{defn}
We say that a covering $G_0 \getsby{\fai_0} G_1 \getsby{\fai_1} G_2 \getsby{\fai_2} \dotsb$ is of {\it Kakutani--Rohlin type} if $G_n$ is a generalized figure-8 and
 $\fai_n(v_{n+1,0}) = v_{n,0}$ for every $n \bni$.
We call a covering of Kakutani--Rohlin type as a {\it {\rm KR}-covering} in short.
A {\rm KR}-covering has {\it rank} $1 \le K \le \infty$ if $\limsup_{n \to \infty}l_n = K$.
\end{defn}
%
%
Note that we can write for each $n > 0$ and $i$ with $1 \le i \le l_{n}$,
\[\fai_{n-1}(c_{n,i}) = c_{n-1,a(n,i,1)}c_{n-1,a(n,i,2)}\dotsb c_{n-1,a(n,i,k(n,i))}.\]
Furthermore, because of the \pdirectional ity of a graph cover, we get 
 \[a(n,1,1) = a(n,2,1) = \dotsb = a(n,l_n,1).\]
For an $n \bni$, $\fai_n$ is \bidirectional\ if and only if
 \[a(n,1,k(n,1)) = a(n,2,k(n,2)) = \dotsb = a(n,l_n,k(n,l_n)).\]
%
%
\begin{rem}\label{rem:exist-kr-partition}
For a KR-covering $G_0 \getsby{\fai_0} G_1 \getsby{\fai_1} G_2 \getsby{\fai_2} \dotsb$,
 we write $G_{\infty} = (X,f)$.
Although, $f$ may not be a homeomorphism, we are able to 
 consider a refining sequence of Kakutani--Rohlin partitions.
For each $n \ge 0$ and $c_{n,i}$ with $1 \le i \le l_n$, 
 let $c_{n,i} = (v_0,v_1, \dotsc, v_{l(n,i)} = v_0)$.
We write $B_i = f^{-1}(U(v_1)) \subset U(v_0)$.
Then, $f^{k}(B_i) = U(v_k)$ for $1 \le k < l(n,i)$.
We write $\xi_i = \seb f^k(B_i) \mid 0 \le k < l(n,i) \sen$.
The zero-dimensional system $(X,f)$ is partitioned into
 Kakutani--Rohlin partition $\Xi_n :=\bigcup_{1 \le i \le l_n} \xi_i$.
It is clear that $\Xi_{n+1}$ refines $\Xi_n$ for all $n \ge 0$.
\end{rem}
%
%
%
%
\begin{rem}\label{rem:a-link}
A KR-covering is linked with
 an ordered Bratteli diagram (that may not be properly ordered nor simple).
To see this, let
 $G_0 \getsby{\fai_0} G_1 \getsby{\fai_1} G_2 \getsby{\fai_2} \dotsb$ be a
 KR-covering.
The ordered Bratteli diagram $(V,E)$ is constructed as follows:
if $G_n$ consists of circuits $c_{n,i}$ $(1 \le i \le l_n)$,
 then we define $V_n = \seb c_{n,i} \mid 1 \le i \le l_n \sen$ for each $n \ge 0$;
if $\fai_{n-1}(c_{n,i}) = c_{n-1,a(n,i,1)}c_{n-1,a(n,i,2)}\dotsb c_{n-1,a(n,i,k(n,i))}$,
an edge that belongs to $E_{n}$ is made from $c_{n,i}$ to each $c_{n-1,a(n,i,j)}$
 for $1 \le j \le k(n,i)$ that is numbered by $j$.
Conversely, a properly ordered Bratteli--Vershik system has Kakutani--Rohlin partitions.
Therefore, we can get a KR-covering with respect to this partitions.
\end{rem}
\begin{thm}\label{thm:kr-covering}
Let $G_0 \getsby{\fai_0} G_1 \getsby{\fai_1} G_2 \getsby{\fai_2} \dotsb$ be a
 {\rm KR}-covering of rank $1 \le K \le \infty$, and $(X,f)$ be its inverse limit.
Then, there exists an $x_0 \in X$ such that for all $x \ne x_0$,
 it follows that $\hash f^{-1}(x) = 1$, and $\hash f^{-1}(x_0) \le K$.
\end{thm}
\begin{proof}
Let $v_{n,0}$ be the central vertex of $G_n$ for each $n \ge 0$.
We define as $x_0 := (v_{0},v_{1,0},v_{2,0},\dotsc)$.
Let $x \ne x_0$ and $x = (v_{0},u_1,u_2,\dotsc)$.
Because $x \ne x_0$, there exists an $N > 0$
 such that $u_n \ne v_{n,0}$ for all $n \ge N$.
Take a $y \in X$ such that $f(y) = x$.
We write $y = (v_0,v_1,v_2,\dotsc)$.
It follows that $(v_n,u_n)$ is an edge of a circuit of $G_n$.
This circuit cannot be $(v_{n,0},v_{n,0})$ for $n \ge N$.
Thus, $v_n$ is uniquely determined for all $n \bni$.
Therefore, $y$ is uniquely determined, i.e., $\hash f^{-1}(x) = 1$.
For the number $\hash f^{-1}(x_0)$, it is sufficient to consider the case in which
 $K$ is finite.
Suppose that there exit mutually distinct $x_1,x_2,\dotsc,x_{K+1}$ such that
 $f(x_i) = x_0$ for all $1 \le i \le K+1$.
For each $1 \le i \le K+1$, let $x_i = (v_0,v_{1,i},v_{2,i},\dotsc)$.
Then, there exists a pair $i \ne i'$ with $1 \le i, i' \le K+1$
 such that for infinitely many $n$, $v_{n,i} = v_{n,i'}$.
This implies that $x_i = x_{i'}$, a contradiction.
\end{proof}
Here, we introduce a proposition that describes a condition of minimality
 of a graph covering.
\begin{prop}\label{prop:minimal}
Let $G_0 \getsby{\fai_0} G_1 \getsby{\fai_1} G_2 \getsby{\fai_2} \dotsb$ be a covering.
Then, the resulting zero-dimensional system $G_{\infty}$ is minimal
 if and only if for all $n \ge 0$,
 there exists an $m > n$ such that $V(\fai_{m,n}(c_{m,i})) = V(G_n)$.
\end{prop}
\begin{proof}
By \cite[(a),(d),(e) of Theorem 3.5]{Shimomura5}, the conclusion is obvious.
\end{proof}
With this proposition,
 when we consider a minimal zero-dimensional system, telescoping if necessary, 
 we can assume that
 $V(\fai_{n-1}(c_{n,i})) = V(G_{n-1})$ and $E(\fai_{n-1}(c_{n,i})) = E(G_{n-1})$
 for all $n > 0$ and $1 \le i \le l_n$,
 i.e. all circuit $c_{n-1,j}$ $(1 \le j \le l_{n-1})$ appear in every 
 $\fai_{n-1}(c_{n,i})$ with $1 \le i \le l_{n}$.
\begin{defn}
A KR-covering $G_0 \getsby{\fai_0} G_1 \getsby{\fai_1} G_2 \getsby{\fai_2} \dotsb$
 is said to be {\it simple} if for every $n \bni$, there exists an $m > n$ such that
 $V(\fai_{m,n}(c_{m,i})) = V(G_n)$.
\end{defn}
%
%
\begin{defn}
A Cantor minimal system that is an inverse limit of some KR-covering of finite rank
 has {\it topological {\rm KR}-rank} $K$ if 
 it is represented by some KR-covering of rank $K$ and $K$ is the least such number.
\end{defn}
The notion of KR-covering is indispensable to state a partial answer for a problem
 that appears in later section.
\section{Covering of Gambaudo--Martens Type}\label{sec:gambaudo-martens}
%
%
In Theorem 2.5 of \cite{GM},
 Gambaudo and Martens showed that every Cantor minimal system is 
an inverse limit of a special kind of graph covering.
In our context, their construction of graph covering is as follows:
 let $G_0 \getsby{\fai_0} G_1 \getsby{\fai_1} G_2 \getsby{\fai_2} \dotsb$
 be a graph covering.
As usual, we assume that $G_0$ is the singleton graph
 $(\seb v_0 \sen,\seb (v_0,v_0) \sen)$.
We shall construct graphs $G_n$ with an $n \ge 1$ such that
 there exist a unique vertex $v_{n,0}$ and finite number of circuits $c_{n,i}~(1 \le i \le l_n)$
 that starts and ends at $v_{n,0}$;
 and, roughly, if two circuits meet at a vertex, then the rest 
 of the two circuits merge until they reach to the end.
\begin{defn}\label{defn:GM-covering}
We say that
 a covering $G_0 \getsby{\fai_0} G_1 \getsby{\fai_1} G_2 \getsby{\fai_2} \dotsb$
is of {\it Gambaudo--Martens type} if for each $n > 0$, there exist 
 a vertex $v_{n,0}$ and finite number of circuits
 $c_{n,i}~(1 \le i \le l_n)$ and a covering map $\fai_n$ such that
\enumb
\item $c_{n,i}$ can be
 written as
 $(v_{n,0}=v_{n,i,0},v_{n,i,1},v_{n,i,2},\dotsc,v_{n,i,l(n,i)} = v_{n,0})$
 with $l(n,i) \ge 1$,
\item $\bigcup_{i = 1}^{l_n}E(c_{n,i}) = E(G_n)$,
\item if $v_{n,i,j} = v_{n,i',j'}$ with $j,j' \ge 1$, then
 $v_{n,i,j+k} = v_{n,i',j'+k}$
 for $k = 0,1,2,\dotsc$, until $j+k = l(n,i)$ and $j'+k = l(n,i')$
 at the same time,
\item $\fai_{n}(v_{n+1,0}) = v_{n,0}$, and 
\item $\fai_n(v_{n+1,i,1}) = v_{n,1,1}$ for all $1 \le i \le l_n$.
\enumn
We say that a covering of this type is a {\it {\rm GM}-covering} in short.
An GM-covering is said to be {\it simple} if 
 for all $n > 0$, there exists an $m > n$ such that
 for each $1 \le i \le l_m$, $E(\fai_{m,n}(c_{m,i})) = E(G_n)$.
As listed in the former Proposition \ref{prop:minimal}, 
 this condition makes the resulting zero-dimensional system minimal.
If we want to avoid the case in which the resulting zero-dimensional system
has an isolated points, we have to add the following condition:
 for every $n \ge 1$ and every vertex $v$ of $G_n$,
 there exist an $m > n$ and distinct vertices $u_1,u_2$ of $G_m$ such that
 $\fai_{m,n}(u_1) = \fai_{m,n}(u_2) =v$.
The {\it rank} of a GM-covering is 
 the integer $1 \le K \le \infty$ defined by $K := \liminf_{n \to \infty}l_n$.
\end{defn}
\begin{rem}
For $n \ge 0$ and $1 \le i \le l_{n+1}$, we can write as:
\[\fai_n(c_{n+1,i}) = c_{n,a(n,i,1)}c_{n,a(n,i,2)}\dotsb c_{n,a(n,i,k(n,i))},\]
such that $a(n,i,1) = 1$ for all $i$ with $1 \le i \le l_{n+1}$.

\end{rem}

%
%
\begin{nota}\label{nota:GMstep1}
By telescoping, we can make a simple GM-covering to have the condition as follows:
for every $n > 0$ and every  $i~(1 \le i \le l_n)$,
 $E(\fai_n(c_{n+1,i})) = E(G_n)$.
Hereafter, if we say that a GM-covering is simple, we assume that this condition
 is satisfied.
\end{nota}
\begin{thm}[Gambaudo and Martens, \cite{GM}]
A zero-dimensional system is minimal (not necessarily homeomorphic)
 if and only if it is represented as the inverse limit of a simple GM-covering.
\end{thm}
\begin{proof}
See the proof of Theorem 2.5 of \cite{GM}.
\end{proof}
%
%
\begin{rem}
Unlike KR-covering, GM-covering may not have Kakutani--Rohlin decomposition.
Therefore, it may not be possible to link with the Bratteli diagrams.
\end{rem}
As an analogue of topological rank for Cantor minimal homeomorphisms,
 we say that a minimal zero-dimensional system has a {\it topological rank} $K$,
 if there exists a simple GM-covering of rank $K$
 and $K$ is the smallest such integer (see \cite{DM}).
In \cite{DM}, they showed that a Cantor minimal homeomorphism whose topological
rank $K$ is finite and $K > 1$ is expansive,
 i.e., topologically conjugate to a minimal two-sided subshift.
We show an analogous theorem for (not necessarily homeomorphic)
 Cantor minimal systems.
%

Suppose that a simple GM-covering
 $G_0 \getsby{\fai_0} G_1 \getsby{\fai_1} G_2 \getsby{\fai_2} \dotsb$
 produces a minimal zero-dimensional system $G_{\infty}$.
We write $G_{\infty}  = (X,f)$.
We assume that $(X,f)$ is not a single periodic orbit.
Then, because of minimality, $(X,f)$ is a Cantor system
 and has no periodic orbits.
Therefore, the minimal length of circuits of $G_n$ grows to infinity, i.e.
 we get $l(n,i) \to \infty$ uniformly as $n \to \infty$.
\begin{nota}\label{nota:natural-extension}
For $(X,f)$, we construct the natural extension $(\hX_f,\sigma)$ as follows:
\itemb
\item $\hX_f := \{ (\dotsc,x_{-1},x_0,x_{1},x_{2}, \dotsc) \in X^{\Z} ~|~ f(x_{i}) = x_{i+1} \ \ \myforall\  i \bi\}$;
\item for $\hx = (\dotsc,x_{-1},x_0,x_{1},x_{2}, \dotsc) \in \hX_f$,
 $\sigma$ shifts $\hx$ to the left, i.e. $(\sigma(\hx))_i = x_{i+1}$ for all $i \bi$.
\itemn
\end{nota}
It is easy to check that if $(X,f)$ is minimal, then $(\hX_f,\sigma)$ is minimal.
For an $\hx \in \hX_f$ and an $i \bi$, we denote $\hx(i) = x_i$.
Then, $(\sigma(\hx))(i) = x_{i+1}$ for all $i \bi$.
We use a lot of notations and ideas that is found in \cite{DM}.
For every $\hx \in \hX_f$ and $i \bi$,
 there exists a unique $u_{n,i} \in V(G_n)$ such that $x_i \in U(u_{n,i})$.
Therefore, a unique sequence $\hx|_n  := (\dotsc,u_{n,-2},u_{n,-1},u_{n,0},u_{n,1},\dotsc)$ of vertices of $G_n$ is defined such that
 $x_i \in U(\hx|_n(i))$.
Although the vertex $u_{n,i}$ is uniquely determined for each $\hx$,
 $n \ge 0$ and $i \bi$,
 the circuit $c_{n,t}$ with $u_{n,i} \in V(c_{n,t})$ may not be unique.
Nevertheless, if $x_i \in U(v_{n,0})$ for some $i \bi$,
 then there exists a unique $t~(1 \le t \le l_n)$ such that
 $x_{i+1} \in U(v_{n,t,1})$;
 and therefore, $x_{i+j} \in U(v_{n,t,j})$ for all $0 \le j \le l(n,t)$.
%
%
\begin{figure}
\begin{center}\leavevmode 
\xy
(0,18)*{}; (100,18)*{} **@{-},
 (48,16)*{c_{n,1}
 \hspace{6mm} c_{n,3} \hspace{6mm}
 \hspace{12mm} c_{n,1} \hspace{8mm}
 \hspace{5mm} c_{n,3} \hspace{5mm}
 \hspace{6mm} c_{n,2} \hspace{6mm}
 \hspace{4mm} c_{n,1}
 \hspace{2mm} },
(7,18)*{}; (7,14)*{} **@{-},
(28,18)*{}; (28,14)*{} **@{-},
(49,18)*{}; (49,14)*{} **@{-},
(70,18)*{}; (70,14)*{} **@{-},
(84,18)*{}; (84,14)*{} **@{-},
(0,14)*{}; (100,14)*{} **@{-},
 (55,12)*{c_{n+1,5}
 \hspace{36mm} c_{n+1,1} \hspace{6mm}
 \hspace{20mm} c_{n+1,3} \hspace{8mm} },
(28,14)*{}; (28,10)*{} **@{-},
(84,14)*{}; (84,10)*{} **@{-},
(0,10)*{}; (100,10)*{} **@{-},
\endxy
\end{center}
\caption{$n$ and $(n+1)$th rows of a linked array system with cuts.}\label{array-system}
\end{figure}
\begin{nota}
We write this $t$ as $t(\hx,n,i)$, and $c_{n,t}$ as $c(\hx,n,i)$,
 for all $n \ge 0$ and $i \bi$.
\end{nota}
Let $k(0) \bi$ be such that $x_{k(0)} \in U(v_{n,0})$, and
 $k(1) > k(0)$ be the least $k > k(0)$ such that $x_k \in U(v_{n,0})$.
Then, we have combined the interval $u_{n,k(0)},u_{n,k(0)+1},\dotsc,u_{n,k(1)-1}$
 with the unique circuit $c(\hx,n,i)$ with $k(0) \le i < k(1)$.
Thus, we have gotten a sequence of $c_{n,i}$s, and we denote it as $\hx[n]$.
We write $\hx[n](i) = c(\hx,n,i)$ for all $n \ge 0$ and all $i \bi$.
Note that we can recover the sequence of 
 vertices of $G_n$ from $\hx[n]$.
A $\hx[0]$ becomes just a sequence of $v_0$.
%
%
For an interval $[n,m]$ with $m > n$, the combination of rows $\hx|_{n'}$
 with $n \le n' \le m$ is denoted as $\hx|_{[n,m]}$, and
 the combination of rows $\hx[n']$
 with $n \le n' \le m$ is denoted as $\hx[n,m]$, 
The {\it array system} of $\hx$ is the infinite combination $\hx|_{[0,\infty)}$
 of all rows $\hx|_n$ $0 \le n < \infty$.
The {\it linked array system} of $\hx$ is the infinite combination $\hx[0,\infty)$
 of all rows $\hx[n]$ $0 \le n < \infty$.
Note that from the information of $\hx[0,\infty)$, we can recover $\hx|_{[0,\infty)}$
 and can also identify $\hx$ itself.
If each circuit of $G_n$ is considered to be just an alphabet,
 then for $n \ge 0$ and $I < J$, we can consider a finite sequence of circuits of $G_n$:
\[\hx[n](I),\hx[n](I+1), \dotsc, \hx[n](J)\]
even if the completion of the circuits are cut off
 at the right or the left end in the above sequence.
\begin{defn}\label{defn:array-systems}
Let $\ddX_f := \seb \hx|_{[0,\infty)} \mid \hx \in \hX_f \sen$ be a set of sequences of
 symbols that are vertices of $G_n$ $(0 \le n < \infty)$.
The topology is generated by {\it cylinders} such that for $\hx \in \hX_f$ and $N,I > 0$:
\[\ddC(\hx,N,I) :=
 \seb \hy|_{[0,\infty)} \mid \hy \in \hX_f,~\hy|_{[0,N]}(i) = \hx|_{[0,N]}(i) \myforall 
 i \mywith -I \le i \le I \sen.\]
The shift map $\sigma : \ddX_f \to \ddX_f$ is define as above.
Then, $(\ddX_f,\sigma)$ is a zero-dimensional system and we call it
 as an {\it array system} of $(\hX_f,\sigma)$.
Let $\barX_f := \seb \hx[0,\infty) \mid \hx \in \hX_f \sen$ be a set of sequences of
 symbols that are circuits of $G_n$ $(0 \le n < \infty)$.
The topology is generated by {\it cylinders} such that for $\hx \in \hX_f$ and $N,I > 0$:
\[\barC(\hx,N,I) :=
 \seb \hy[0,\infty) \mid \hy \in \hX_f,~\hy[0,N](i) = \hx[0,N](i) \myforall 
 i \mywith -I \le i \le I \sen.\]
The shift map $\sigma : \barX_f \to \barX_f$ is define as above.
Then, $(\barX_f,\sigma)$ is a zero-dimensional system and we call it
 as an {\it linked array system} of $(\hX_f,\sigma)$.
\end{defn}

\begin{rem}
Clearly, $(\ddX_f,\sigma)$ is topologically conjugate to $(\hX_f,\sigma)$.
Moreover, because $(\ddX_f,\sigma)$ has a continuous factor map to $(\barX_f,\sigma)$
 and it is bijective as described above, $(\barX_f,\sigma)$ is also topologically
 conjugate to $(\hX_f,\sigma)$.
\end{rem}
%
%
%
%
For each sequence $\hx[n]$ of circuits of $G_n$ $(n > 0)$,
 we make an {\it $n$-cut at the position $i \bi$}
 when $\hx|_n(i) = v_{n,0}$, i.e., if there exists an $n$-cut at position $i \bi$,
 then $c(\hx,n,i-1)$ and $c(\hx,n,i)$ are separated by the cut.
The row $\hx[n]$ is separated into circuits exactly by the cuts
 (see Figure \ref{array-system}).
Note that for $m > n$, if there exists an $m$-cut at position $i$, then
 there exists an $n$-cut at position $i$.
%
%
%
%
\begin{figure}
\begin{center}\leavevmode 
\xy
(0,22)*{}; (100,22)*{} **@{-},
 (50,20)*{v_0\phantom{{}_{{},f}}\ v_0\phantom{{}_{{},l}}\ v_0\phantom{{}_{{},l}}\ v_0\phantom{{}_{{},l}}\ v_0\phantom{{}_{{},f}}\ v_0\phantom{{}_{{},l}}\ v_0\phantom{{}_{{},l}}\ v_0\phantom{{}_{{},l}}\ v_0\phantom{{}_{{},f}}\ v_0\phantom{{}_{{},l}}\ v_0\phantom{{}_{{},l}}\ v_0\phantom{{}_{{},l}}\ v_0\phantom{{}_{{},f}}\ v_0\phantom{{}_{{},l}}},
(0,22)*{}; (0,18)*{} **@{-},
(7,22)*{}; (7,18)*{} **@{-},
(14,22)*{}; (14,18)*{} **@{-},
(21,22)*{}; (21,18)*{} **@{-},
(28,22)*{}; (28,18)*{} **@{-},
(35,22)*{}; (35,18)*{} **@{-},
(42,22)*{}; (42,18)*{} **@{-},
(49,22)*{}; (49,18)*{} **@{-},
(56,22)*{}; (56,18)*{} **@{-},
(63,22)*{}; (63,18)*{} **@{-},
(70,22)*{}; (70,18)*{} **@{-},
(77,22)*{}; (77,18)*{} **@{-},
(84,22)*{}; (84,18)*{} **@{-},
(91,22)*{}; (91,18)*{} **@{-},
(98,22)*{}; (98,18)*{} **@{-},
(0,18)*{}; (100,18)*{} **@{-},
 (49,16)*{c_{1,3}
 \hspace{8mm} c_{1,3} \hspace{6mm}
 \hspace{11mm} c_{1,1} \hspace{8mm}
 \hspace{7mm} c_{1,3} \hspace{5mm}
 \hspace{6mm} c_{1,2} \hspace{6mm}
 \hspace{4mm} c_{1,1}
 \hspace{2mm} },
(7,18)*{}; (7,14)*{} **@{-},
(28,18)*{}; (28,14)*{} **@{-},
(49,18)*{}; (49,14)*{} **@{-},
(70,18)*{}; (70,14)*{} **@{-},
(84,18)*{}; (84,14)*{} **@{-},
(0,14)*{}; (100,14)*{} **@{-},
(56,12)*{c_{2,3} \hspace{0.7mm}\ 
 \hspace{32mm} c_{2,1} \hspace{27mm}
 \hspace{10mm} c_{2,3}
 \hspace{2mm} },
(28,14)*{}; (28,10)*{} **@{-},
(84,14)*{}; (84,10)*{} **@{-},
(0,10)*{}; (100,10)*{} **@{-},
(36,8)*{c_{3,3} \hspace{26mm}
 \hspace{15mm} c_{3,1}},
(28,10)*{}; (28,6)*{} **@{-},
(0,6)*{}; (100,6)*{} **@{-},
(50,4)*{\vdots},
\endxy
\end{center}
\caption{The first 4 rows of an array system.}\label{array-system-2}
\end{figure}
%
%
%
%
%
For each circuit $c_{n,i}$, we can determine a series of circuits
 by $\fai_{n-1}(c_{n,i}) = c_{n-1,1}c_{n-1,a(n,i,2)}\dotsb c_{n-1,a(n,i,k(n,i))}$.
Furthermore, each $c_{n-1,a(n,i,j)}$ determines a series of circuits by the map 
 $\fai_{n-2}$.
In this way we can determine a set of circuits arranged in a square form as in
 Figure \ref{2-symbol}.
Following \cite{DM},
 this form is said to be the {\it $n$-symbol} and denoted also by $c_{n,i}$.
For $m < n$, the projection $c_{n,i}[m]$ that is a finite sequence of circuits of $G_m$
 is also defined.
\begin{figure}
\begin{center}\leavevmode 
\xy
(28,32)*{}; (84,32)*{} **@{-},
(57,30)*{v_0\phantom{{}_{{},l}}\ v_0\phantom{{}_{{},l}}\ v_0\phantom{{}_{{},l}}\ v_0\phantom{{}_{{},l}}\ v_0\phantom{{}_{{},f}}\ v_0\phantom{{}_{{},l}}\ v_0\phantom{{}_{{},l}}\ v_0\phantom{{}_{{},l}}},
(28,32)*{}; (28,28)*{} **@{-},
(35,32)*{}; (35,28)*{} **@{-},
(42,32)*{}; (42,28)*{} **@{-},
(49,32)*{}; (49,28)*{} **@{-},
(56,32)*{}; (56,28)*{} **@{-},
(63,32)*{}; (63,28)*{} **@{-},
(70,32)*{}; (70,28)*{} **@{-},
(77,32)*{}; (77,28)*{} **@{-},
(84,32)*{}; (84,28)*{} **@{-},
(28,28)*{}; (84,28)*{} **@{-},
 (60,26)*{
 \hspace{7mm} c_{1,1} \hspace{8mm}
 \hspace{7mm} c_{1,3} \hspace{6mm}
 \hspace{6mm} c_{1,2} \hspace{8mm}
 \hspace{2mm} },
(28,28)*{}; (28,24)*{} **@{-},
(49,28)*{}; (49,24)*{} **@{-},
(70,28)*{}; (70,24)*{} **@{-},
(84,28)*{}; (84,24)*{} **@{-},
(28,24)*{}; (84,24)*{} **@{-},
(60,22)*{
 \hspace{23mm} c_{2,1} \hspace{23mm}
 \hspace{2mm} },
(28,24)*{}; (28,20)*{} **@{-},
(84,24)*{}; (84,20)*{} **@{-},
(28,20)*{}; (84,20)*{} **@{-},
\endxy
\end{center}
\caption{The $2$-symbol corresponding to the circuit $c_{2,1}$ of Figure \ref{array-system-2}.}\label{2-symbol}
\end{figure}

%
%
We have to make it clear that the author have borrowed the main story of the proofs
 from \cite{DM}.
Our work was just to check the proof in \cite{DM} is valid in our case of 
 Cantor minimal continuous surjections.
%
%

%
%
It is clear that $\hx[n] = \hx'[n]$ implies $\hx[0,n] = \hx'[0,n]$.
If $\hx \ne \hx'$ $(\hx, \hx' \in \hX_f)$,
 then there exists an $n > 0$ with $\hx[n] \ne \hx'[n]$.
For $\hx, \hx' \in \hX_f$, we say that the pair $(\hx,\hx')$ is {\it $n$-compatible}
 if $\hx[n] = \hx'[n]$.
If $\hx[n] \ne \hx'[n]$, then we say that $\hx$ and $\hx'$ are {\it $n$-separated}.
We recall that
 if there exists an $n$-cut at position $k$, then there exists an $m$-cut at position
 $k$ for all $0 \le m \le n$.
Let $\hx \ne \hx'$. 
%
%
If a pair $(\hx,\hx')$ is $n$-compatible and $(n+1)$-separated,
 then, we say the {\it depth of compatibility} of $x$ and $x'$ is $n$,
 or the pair $(x,x')$ has {\it depth} $n$.
If $(\hx,\hx')$ is a pair of depth $n$ and $(\hx,\hx'')$ is a pair of depth $m > n$,
 then the pair $(x',x'')$ has depth $n$ (and hence never equal).
%
%
An $n$-separated pair $(\hx,\hx')$ is said to {\it have a common $n$-cut} if 
 both $\hx$ and $\hx'$ have an $n$-cut at the same position.
If a pair has a common $n$-cut then it also has a common $m$-cut
 for each $0 \le m \le n$.
The set $X_n := \seb \hx[n] \mid \hx \in \hX_f \sen$ is a two-sided subshift
 of finite set $\seb c_{n,1}, c_{n,2}, \dotsc, c_{n,l_n} \sen$.
The factoring map is denoted by $\pi_n : \hX_f \to X_n$,
 and the shift map is denoted by $\sigma_n : X_n \to X_n$.
We write just $\sigma = \sigma_n$ for all $n$, 
 if there is no confusion.


\section{Main Theorem.}
%
%
In this section, we state our main result, and prove the theorem.
Again, we have to make it clear that the author have borrowed the main story
 of the proofs from \cite{DM}.
Our work from now on is just to check the proofs in \cite{DM} is valid in our case of 
 Cantor minimal continuous surjections.
\begin{thm}[Main Result]
Let $(X,f)$ be a minimal zero-dimensional system
 whose topological rank is $K~(1 < K < \infty)$.
Then, its natural extension $(\hX_f,\sigma)$ is expansive.
\end{thm}
\noindent {\it Proof.}
As in the proof of \cite{DM}, we show by contradiction.
Suppose that the claim fails.
Then, for all $L > 0$, there exists a pair $(\hx,\hx')$ with distinct elements
 of $\hX_f$ that is $L$-compatible.
Because $\hx \ne \hx'$, for some $m > L$, $(\hx,\hx')$ is $m$-separated.
Therefore, $(\hx,\hx')$ has depth $n$ with $L \le n < m$.
Therefore, for infinitely many $n$, there exists a pair $(\hx_n,\hx'_n)$ of depth $n$.
By telescoping, we can assume that
 every $n > 0$ has a pair $(\hx_n,\hx'_n)$ of depth $n$.
Note that even after another telescoping, this quality still holds.
As in the proof of \cite{DM}, we prove in separated cases:
\enumb
\renewcommand{\labelenumi}{(\arabic{enumi})}
\renewcommand{\theenumi}{(\arabic{enumi})}
\setcounter{enumi}{0}
\item\label{case1} there exists an $N$ such that for all $n > N$ and every $m > n$,
 there exists a pair $(\hx_n,\hx'_n)$ of depth $n$ with a common $m$-cut;
\item\label{case2} for infinitely many $n$, and then every sufficiently large
 $m > n$, any pair of depth $n$ has no common $m$-cut.
\enumn

\vspace{3mm}

\noindent {\it Proof in case \ref{case1}.}
As in the proof of \cite{DM},
 we prove that such a case never occurs, even for $K = 1$.
Fix some $m > N + K$ and for an integer $n \in [m - K, m -1]$ let $(\hx_n,\hx'_n)$ be
 a pair of depth $n$ with a common $m$-cut.
For $n = m -1$, we have a $(m - 1)$-compatible $m$-separated pair
 $(\hx_{m-1},\hx'_{m-1})$ with a common $m$-cut.
%
%
%
Suppose that all of the $m$-cuts of $\hx_{m-1}$ and $\hx'_{m-1}$ are same.
Because the pair $(\hx_{m-1},\hx'_{m-1})$ is $m$-separated, 
 at least two distinct symbols are in the same place of $\hx_{m-1}$ and $\hx'_{m-1}$.
These symbols have the same rows from $0$ to $m-1$.
Therefore, at least two distinct $m$-symbols have the same rows from $0$ to $m-1$.
If a set of $m$-symbols whose rows from $0$ to $m-1$ are the same,
 we factor these to the same alphabet,
 i.e. we make a new GM-covering identifying such circuits of $G_m$.
The new directed graph is also a GM-covering,
 the projection mapping to $G_{m-1}$ from the new $G_m$ is well defined, and
 the projection mapping to the new $G_{m}$ is also well defined.
It is clear that the new covering is \pdirectional.
Furthermore, the inverse limit of the new covering and the $X_{m-1}$ of the new covering is the same as the old one.
Because the pair $(\hx_{m-2},\hx'_{m-2})$ was $(m-1)$-separated, and
 this factorization does not affect in the $(m-1)$th row,
 the pair $(\hx_{m-2},\hx'_{m-2})$ is still $m$-separated.
Note that the number of circuits of $m$th row has been decreased at least $1$.
%
%
Next, we consider the case in which after a common $m$-cut of 
 the pair $(\hx_{m-1},\hx'_{m-1})$,
 the coincidence of the positions of $m$-cuts do not continue to the right or left end.
\begin{figure}
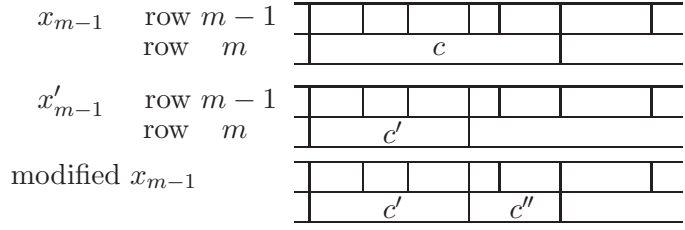

\begin{center}\leavevmode 
\xy
(33,28)*{}; (84,28)*{} **@{-},
(15,26)*{x_{m-1} \hspace{4mm} \text{ row } m-1},
(60,26)*{},
(35,28)*{}; (35,24)*{} **@{-},
(42,28)*{}; (42,24)*{} **@{-},
(48,28)*{}; (48,24)*{} **@{-},
(56,28)*{}; (56,24)*{} **@{-},
(60,28)*{}; (60,24)*{} **@{-},
(68,28)*{}; (68,24)*{} **@{-},
(80,28)*{}; (80,24)*{} **@{-},
%
(33,24)*{}; (84,24)*{} **@{-},
(20,22)*{\text{row \hspace{3mm}} m },
(60,22)*{
 \hspace{16mm} c \hspace{30mm}
 \hspace{2mm} },
(35,24)*{}; (35,20)*{} **@{-},
(68,24)*{}; (68,20)*{} **@{-},
%
(33,20)*{}; (84,20)*{} **@{-},
(33,17)*{}; (84,17)*{} **@{-},
(15,15)*{x'_{m-1} \hspace{4mm} \text{ row } m-1},
(60,13)*{},
(35,17)*{}; (35,13)*{} **@{-},
(42,17)*{}; (42,13)*{} **@{-},
(48,17)*{}; (48,13)*{} **@{-},
(56,17)*{}; (56,13)*{} **@{-},
(60,17)*{}; (60,13)*{} **@{-},
(68,17)*{}; (68,13)*{} **@{-},
(80,17)*{}; (80,13)*{} **@{-},
%
(33,13)*{}; (84,13)*{} **@{-},
(20,11)*{\text{row \hspace{3mm}} m },
(60,11)*{
 \hspace{10mm} c' \hspace{36mm}
 \hspace{2mm} },
(35,13)*{}; (35,9)*{} **@{-},
(56,13)*{}; (56,9)*{} **@{-},
%
(33,9)*{}; (84,9)*{} **@{-},
(33,7)*{}; (84,7)*{} **@{-},
(10,5)*{\text{modified } x_{m-1} \hspace{4mm}},
(60,3)*{},
(35,7)*{}; (35,3)*{} **@{-},
(42,7)*{}; (42,3)*{} **@{-},
(48,7)*{}; (48,3)*{} **@{-},
(56,7)*{}; (56,3)*{} **@{-},
(60,7)*{}; (60,3)*{} **@{-},
(68,7)*{}; (68,3)*{} **@{-},
(80,7)*{}; (80,3)*{} **@{-},
%
(33,3)*{}; (84,3)*{} **@{-},
%
(60,1)*{
 \hspace{10mm} c' \hspace{36mm}
 \hspace{2mm} },
(77,1)*{
 \hspace{10mm} c'' \hspace{36mm}
 \hspace{2mm} },
(35,3)*{}; (35,-1)*{} **@{-},
(56,3)*{}; (56,-1)*{} **@{-},
(68,3)*{}; (68,-1)*{} **@{-},
%
(33,-1)*{}; (84,-1)*{} **@{-},
\endxy
\end{center}
\caption{Change of an $m$-symbol. There is a common cut at the left end.}\label{change-symbol}
\end{figure}
Suppose that after a common $m$-cut at position $k_0$,
 the continuation of common $m$-cuts does not continue to the right end.
Let $k_1 > k_0$ be the position of the first common $m$-cut
 such that the right $m$-symbols $c$ and $c'$ of $\hx_{m-1}$ and $\hx'_{m-1}$
 has different length.
We assume without loss of generality that $l(c') < l(c)$.
Let $k_2 = k_1 +l(c')$.
Then, $k_2$ is the position of the next $m$-cut of $\hx'_{m-1}$.
Because $\hx_{m-1}[m-1] = \hx'_{m-1}[m-1]$,
 the $m$-symbol $c$ itself has an $(m-1)$-cut at $l(c')$.
Let $c = (v_{n,0},v_1,v_2,\dotsc,v_{l(c)})$.
Then, $c$ is separated into two parts $(v_{m,0},v_1,v_2,\dotsc,v_{l(c')})$
 and $(v_{l(c')},v_{l(c')+1},\dotsc,v_{l(c)})$ with the vertex $v_{l(c')}$ in common.
Let $c= c_1,c_2,\dotsc,c_s$ be the circuits of $G_m$ that passes $v_{l(c')}$.
We make a new graph identifying $v_{l(c')}$ with $v_{m,0}$.
Note that $(c[m-1])(l(c'))$ is the ending position of $c'$.
Therefore, the first circuit that begins from $(c[m-1])(l(c'))$ is $c_{m-1,1}$.
Because of this, the new cover map satisfies \pdirectional ity.
Each circuit $c_i$ $i = 1,2,\dotsc,s$ is decomposed into two circuits
 $\barc_i$ that is distinct for each $i$ and the common $c''$, i.e.
 each original circuit $c_i$ is now a cycle and $c_i = \barc_i c''$.
The number of circuits is now increased $1$.
Nevertheless, for two circuits $\barc_1~({\rm or }~\barc~)$ and $c'$,
 we get $\barc_1[m-1] = c'[m-1]$.
Therefore, we can merge these two circuits and construct a new graph and the
 covering.
For the last covering, the new $G_m$ has the same number of circuits as the
 original one (see Figure \ref{change-symbol}).
Furthermore, they have the same inverse limits and the same $X_{m-1}$.
After this modification no cut that existed is removed.
The number of $m$-symbols is not changed.
The coincidence of the $m$-cut from $k_0$ to the right might be shortened.
Nevertheless, the same modification is possible, and
 finally we come to the point that we have the same $m$-cuts from $k_0$ to the right end.
The same argument is valid to the left direction from $k_0$.
Now, 
 $\hx_{m-1}$ and $\hx'_{m-1}$ have the same $m$-cuts throughout the sequences.
Then, we can apply the previous argument.
In this way, we get a factor in the $m$th row that decreases the number of circuits
 of $G_m$, and yet $(\hx_{m-2},\hx'_{m-2})$ is still $m$-separated.

We can now delete the $(m-1)$th row, and continue this process.
Finally, the $m$th row is represented by only one circuit,
 and still the pair $(\hx_{m-K},\hx'_{m-K})$ is $m$-separated and has a common $m$-cut,
 a contradiction.

\vspace{5mm}

%
%
%
\noindent{\it Proof in case (2)}.
As in the proof of \cite{DM}, we show that $(\hX_f,\sigma)$ becomes an odometer.
For arbitrarily large $n$, there exists an $m(n)>n$ such that
 every pair $(\hx_n,\hx'_n)$ of depth $n$
 has no common $m(n)$-cut.
As described shortly in \cite{DM}, by telescoping,
 we wish to show that the condition of \ref{case2} holds for every $n$.
Fix $n$.
Take an $n' > m(n)$ such that
 for every pair $(\hx_{n'},\hx'_{n'})$ of depth $n'$
 has no common $m'(n')$-cut.
Then, because $n' > m(n)$, every pair $(\hx_n,\hx'_n)$ of depth $n$
 has no common $n'$-cut.
Thus, by telescoping (from $n'$ to $n$),
 for every pair $(\hx_n,\hx'_n)$ of depth $n$ has
 no common $(n+1)$-cut.
Thus, consecutive application of this telescoping,
 we get a covering such that
 for every  $n$, every pair $(\hx_n,\hx'_n)$ of depth $n$
 has no common $(n+1)$-cut.

\begin{lem}\label{lem:translation}
Let $(\hx_n,\hx'_n)$ be a pair of depth $n$ that has no common $(n+1)$-cut.
Let $\hy \in \hX_f$.
Then, there exists a $\hy' \in \hX_f$ such that $\hy[n] = \hy'[n])$
 and the pair $(\hy,\hy')$ has no common $(n+1)$-cut.
Especially, $(\hy,\hy')$ has depth $n$.
\end{lem}
\begin{proof}
By minimality of $(\hX_f,\sigma)$, there exists a sequence $n_1 < n_2 < \dotsb$ such that
 $\lim_{k \to \infty}\sigma^{n_k}(\hx_n) = \hy$.
Taking a subsequence if necessary,
 we also get $\lim_{k \to \infty}\sigma^{n_k}(\hx'_n) = \hy'$ for some $\hy' \in \hX_f$.
Take an $N > n$ arbitrarily.
We get $\pi_N(\sigma^{n_k}(\hx_n)) = \sigma^{n_k}((\hx_n[N])) \to \hy[N]$
 and $\pi_N(\sigma^{n_k}(\hx'_n)) = \sigma^{n_k}((\hx'_n)[N]) \to \hy'[N]$
 as $k \to \infty$.
Because $\sigma^k(\hx)[n] = \sigma^k(\hx')[n]$ for all $k \bi$,
 we get $\hy[n] = \hy'[n]$.
Note that this does not directly imply that the depth of the pair $(\hy,\hy')$ is $n$.
Nevertheless, by the same reason, it follows that $(\hy,\hy')$ has no common $(n+1)$-cut.
Therefore, $(\hy,\hy')$ is $(n+1)$-separated, and $(\hy,\hy')$ has depth $n$.
\end{proof}
Fix some $n_0 \ge 1$ and let $m = n_0 + K^{K+1}$.
For $n \in [n_0,m-1]$, fix a pair $(\hx_n,\hx'_n)$ of depth $n$.
Fix $\hy_0 \in X$.
By Lemma \ref{lem:translation},
 we can take pairs $(\hy_0, \hy_n)$ with $n \in [n_0,m-1]$
 such that $\hy_0[n] = \hy_n[n]$ and has no common $(n+1)$-cut.
Then, for any $n,n' \in [i_0,m-1]$, $n' > n$,
 the pair $(\hy_{n},\hy_{n'})$ has depth $n$ and has no common $(n+1)$-cut.
In this way, we have found $K^{K+1}+1$ number of $n_0$-compatible and pairwise
 $m$-separated elements $\hy_0,\hy_{i_0},\hy_{i_0+1},\dotsc,\hy_{m-1}$ with
 no common $m$-cuts.

The next lemma is said to be the infection lemma in \cite{DM}, with which
 we can show that $X_{n_0}$ is periodic.
Then, because $n_0$ is arbitrary, we can conclude that
 $(\hX_f,\sigma)$ is an odometer system.
\begin{lem}[Downarowicz and Maass, \cite{DM}]\label{lem:infection}
If there exist at least $K^{K+1}+1$ $n$-compatible points
 $\hy_k~(k \in [1,K^{K+1}+1])$, which, for some $m > n$, are pairwise
 $m$-separated with no common $m$-cuts, then $X_n$ is periodic. 
\end{lem}
\begin{proof}
As in the proof in \cite{DM}, denote by $\hy$ the common image of the points
 $\hy_k$ in $X_n$.
Think of the sequences $\hy$ and $\hy_k[m]$'s.
Let us fix a position $k_0$.
Then, for at least $K^{K}+1$ times,
 the lines $\hy_k[m]$ overlap with the same $m$-symbol.
This $m$-symbol is denoted by $c$.
Because, $c$ overlaps at least twice, we can find a periodicity in $c[n]$,
 i.e., there exists an $l > 0$ such that $(c[n])(i + l) = (c[n])(i)$ for
 $0 \le i < l(c) - l$.
Let $l_c$ be the minimum of such number.
Thus, we have the periodicity in $\hy$ around $k_0$ of length at least $l(c)$.
Suppose that this periodicity continues to the right end.
Then, $\hy$ is eventually periodic, and by minimality, $X_n$ turns out to be 
 a periodic orbit as required.
%
%
Therefore, suppose that this periodicity is finitely limited and broken
 for the first time at $k_1$.
We have chosen at least $K^K+1$ lines already, and we get rid of the rest.
These overlap with the same $m$-symbol $c$ at the position $k_0$.
At the position $k_1$, one can get at least $K^{K-1}+1$ lines $y_k[m]$
 with the same $m$-symbol $c'$ overlapping at $k_1$.
Suppose that $c' = c$.
Then, because at least two lines overlaps, at least one line is slided left
 more than $l_c$ from the position $k_1$.
This shows that $\hy(k_1-l_c) = \hy(k_1)$, contradicting the limited periodicity.
Thus, suppose that $c' \ne c$.
%
%
We have the two cases: (a) $l_{c'} \ge l_{c}$, and (b) $l_{c'} < l_c$.
%
%
Suppose that the case (a) occurred.
The lines overlapping at $k_1$ have the same symbol $c$ over $k_0$.
Thus, the left end of $c'$ that occurs over $k_1$, does not exceed $k_0$.
Now, because we have at least $3$ lines overlapping at $k_1$,
 we have two $c'$s slided left from $k_1$.
Therefore, the interval $[k_1-2l_{c'}-l_c,k_1-1]$ has periodicity $l_c$.
Furthermore, for $0 < s \le l_c$, we get $\hy(k_1-s) = \hy(k_1-s-l_{c'})$.
If $l_{c'} = l_{c}$, then the overlapping of $c'$ on $k_1$ implies
 the periodicity $\hy(k_1-l_c) = \hy(k_1)$.
This contradicts the limited periodicity.
Therefore, we suppose that $l_{c'} > l_c$.
Then, $\hy(k_1-l_{c'}+a l_c)$ with $a = 0,1,2,\dotsc$ have the same symbol
 with $\hy(k_1)$ while $-l_{c'}+a l_c < 0$.
Because $l_{c'} > l_{c}$, there exists an $0 < s_1 \le l_c$ and an $a_1 > 0$ such that
 $k_1- s_1 = k_1-l_{c'}+a_1 l_c$,
 and we get $\hy(k_1-s_1) = \hy(k_1-l_{c'}+a_1 l_c) = \hy(k_1)$.
It follows that $k_1 - s_1-l_{c'} \in [k_1-2l_{c'}-l_c,k_1-1]$.
Then, we get $\hy(k_1-s_1) = \hy(k_1-s_1-l_{c'}) = \hy(k_1-2l_{c'}+a_1 l_c)$.
Again,
 we get $\hy(k_1-s_2) = \hy(k_1-s_1-l_{c'}) = \hy(k_1 -2l_{c'}+a_2 l_c) = \hy(k_1)$
 for some $0 < s_2 \le l_c$ and $a_2 > 0$.
In this way, we get a sequence $0 < s_i \le l_c$ and $0 < a_i$
 with $1 \le i < \infty$
 such that $\hy(k_1-s_i) = \hy(k_1-i l_{c'} +a_i l_c) = \hy(k_1)$
 for all $1 \le i < \infty$.
Because $\gcd(l_c,l_{c'}) = s_i$ for some $i > 0$, we also get
 $l_c = s_j$ for some $j > 0$.
Therefore, we have $\hy(k_1-l_c) = \hy(k_1)$,
 contradicting the limited periodicity.
Therefore, if we have case (a), then we have a contradiction.
Suppose that the case (b) occurred.
Then, we restart the argument from the position $k_1$, and continue 
 in the same way with $K^{K-1}+1$ lines.
If the periodicity is right limited, we shall encounter only the case (b).
Further, because the period of the periodicity decreases by the condition (b),
 the right end is covered with an $m$-symbol that has not appeared
 previously.
At the last step, there exists only one symbol left, and there remains
 at least $K+1$ lines.
Therefore, case (a) must happen, the final contradiction.

We have proved that the right end is not limited.
This concludes the proof.
 
\end{proof}
Because the infection lemma is shown, the theorem is proved.
\qed

\begin{rem}
If the natural extension of a Cantor minimal system of finite topological rank has 
 finite topological rank, then our main result becomes a waste of effort.
Nevertheless, we could not answer this question.
\end{rem}
In the rest of this section, we shall show the next theorem
 that gives a partial answer affirmatively.
\begin{thm}\label{thm:kr-extension}
Let $(X,f)$ be a Cantor minimal system with topological rank $1 \le K \le \infty$.
The natural extension of $(X,f)$ is denoted by $(\hX_f,\sigma)$.
Then, there exists a Cantor minimal system of Kakutani--Rohlin type $(X',f')$
 with {\rm KR}-rank $\le K$, such that
 the natural extension $(\hX'_{f'},\sigma)$ is naturally topologically conjugate
 to $(\hX_f,\sigma)$.
Specifically, as a Cantor minimal system of Kakutani--Rohlin type, $(X',f')$ is 
 1-1 except one point $x_0 \in X'$ with which $\hash f'^{-1}(x_0) \le K$.
\end{thm}
\begin{proof}
Let 
 $\Gcal : G_0 \getsby{\fai_0} G_1 \getsby{\fai_1} G_2 \getsby{\fai_2} \dotsb$
 be a simple GM-covering of rank $1 \le K \le \infty$ such that
 the inverse limit is topologically conjugate to $(X,f)$.
As a GM-covering, every $G_n$ $(n > 0)$ consists of circuits
 $\seb c_{n,1}, c_{n,2}, \dotsc, c_{n,l_n} \sen$ such that the central vertex is
 $v_{n,0}$.
Then, as in Notation \ref{nota:natural-extension}, we can define its natural extension
 $(\hX_f,\sigma)$, the array system and the linked array system.
For each $n > 0$, we shall construct another directed graph $G'_n$ and cover maps
 $\fai'_n : G'_{n+1} \to G'_n$.
Let $c_{n,i}$ $(1 \le i \le l_n)$ be circuits of $G_n$; and $v_{n,0}$,
 its central vertex.
Then, as in the remark after the definition of GM-covering, we can write as:
\[\fai_n(c_{n+1,i}) = c_{n,a(n,i,1)}c_{n,a(n,i,2)}\dotsb c_{n,a(n,i,k(n,i))}\]
for each $n \ge 0$ and each $1 \le i \le l_{n+1}$,
 where $a(n,i,1) = 1$ for all $1 \le i \le l_{n+1}$.
Then, we can write
 $c_{n,i} = (v_{n,0}=v_{n,i,0},v_{n,i,1},v_{n,i,2},\dotsc,v_{n,i,l(n,i)} = v_{n,0})$.
It may happen that $v_{n,i,j} = v_{n,i',j'}$ for some $i \ne i'$, $1 \le j < l(n,i)$
 and $1 \le j' < l(n,i')$.
Nevertheless,
 for the construction of $G'_n$,
 we make distinct vertices $v'_{n,i,j}$ and $v'_{n,i',j'}$.
Therefore, we make use of $1+\sum_{1 \le i \le l_n}(l(n,i)-1)$ vertices
\[\seb v'_{n,0} \sen \cup \seb v'_{n,i,j} \mid 1 \le i \le l_n, 1 \le j < l(n,i) \sen.\]
As in $G_n$, we assume that $v'_{n,0} = v'_{n,i,0} = v'_{n,i,l(n,i)}$.
The set of edges is defined as:
\[\seb (v'_{n,i,j},v'_{n,i,j+1}) \mid 1 \le i \le l_n \myand 0 \le j < l(n,i)\sen.\]
Thus, we constructed a new directed graph $G'_n$ that consists of circuits $c'_{n,i}$
 corresponding to $c_{n,i}$ for each $1 \le i \le l_n$.
The graph homomorphisms $\fai'_n : G'_{n+1} \to G'_n$ is defined as:
\itemb
\item $\fai'_n(v'_{n+1,0}) = v'_{n,0}$,
\item $\fai'_n(c'_{n+1,i}) =  c'_{n,a(n,i,1)}c'_{n,a(n,i,2)}\dotsb c'_{n,a(n,i,k(n,i))}$
 for each $1 \le i \le l_{n+1}$.
\itemn
Because $a(n,i,1) = 1$ for all $1 \le i \le l_n$, it is evident that the sequence
 $\Gcal' : G'_0 \getsby{\fai'_0} G'_1 \getsby{\fai'_1} G'_2 \getsby{\fai'_2} \dotsb$
 is a graph covering.
From the construction, this covering is a KR-covering.
We can get its inverse limit $G'_{\infty} = (X',f')$.
It is evident that $(X',f')$ has KR-rank $\le K$.
It is also obvious that $(X',f')$ has a natural factor map onto $(X,f)$.
Then, as in Notation \ref{nota:natural-extension}, we can define its natural extension
 $(\hX'_{f'},\sigma)$, the array system and the linked array system.
Note that the linked array systems of $\Gcal$ and $\Gcal'$ coincide.
As have been shown, the linked array systems have not lost any information of 
 the natural extensions.
Thus, it follows that $(\hX_f,\sigma)$ and $(\hX'_{f'},\sigma)$
 are naturally topologically conjugate.
\end{proof}

\begin{cor}\label{cor:coincidence-rank}
For a Cantor minimal homeomorphism, the topological rank that is defined by Bratteli diagrams and the topological rank that is defined by GM-coverings coincide.
\end{cor}
\begin{proof}
Let $(X,f)$ be a Cantor minimal homeomorphism such that the topological rank with regard
 to the Bratteli diagram representation is $1 \le K \le \infty$.
Then, there exists a properly ordered Bratteli diagram with rank $K$.
Then, as described in Remark \ref{rem:a-link}, $(X,f)$ is an inverse limit of
 a KR-covering with rank $K$.
It is evident that a KR-covering of rank $K$ is a GM-covering of rank $K$.
Therefore, $(X,f)$ has a topological rank $\le K$ with regard to GM-coverings.
On the other hand, suppose that there exists a GM-covering
 with rank $1 \le K' \le \infty$ such that
 the inverse limit is topologically conjugate to $(X,f)$.
Then, there exists a KR-covering with rank $K'$ such that the inverse limit
 is topological conjugate to the natural extension of $(X,f)$, i.e., $(X,f)$ itself.
By telescoping, this KR-covering becomes bidirectional.
Therefore, the corresponding ordered Bratteli diagram has unique maximal infinite path.
The uniqueness of minimal infinite path is evident.
The simplicity of the Bratteli diagram comes from the minimality of the Cantor system.
Therefore, this Bratteli diagram is properly ordered and has rank $\le K'$.
This concludes the coincidence of two definitions of the topological rank.
\end{proof}

%
%
%

\end{document}